\documentclass[12pt]{amsart}

\usepackage{amsmath,amssymb,amsthm}
\usepackage{booktabs}
\usepackage{microtype}
\usepackage{url}

\newtheorem{theorem}{Theorem}
\newtheorem{lemma}[theorem]{Lemma}

\newtheorem{corollary}[theorem]{Corollary}
\theoremstyle{definition}
\newtheorem{definition}[theorem]{Definition}
\theoremstyle{remark}
\newtheorem{remark}[theorem]{Remark}

\title{The Collision Invariant}

\author{Alexander S.\ Petty}
\email{alexander.petty@gmail.com}
\date{March 2026}

\begin{document}
\begin{abstract}
For a prime $p$ and base $b$, the digit function
$\delta(r) = \lfloor br/p \rfloor$ partitions the residues
$\{1, \ldots, p{-}1\}$ into $b$ contiguous bins. The
\emph{collision count} $C(g)$ records how many residues
share a bin with their image under multiplication by~$g$.
We prove four results about this function.

First, the \emph{gate width theorem}: exactly $b - 1$
multipliers satisfy $C(g) = 0$, given by the explicit
family $g = -u/(b{-}u) \bmod p$ for $u = 1, \ldots, b{-}1$.
The proof rests on a linearization that transforms the
contiguous Beatty bins into residue classes modulo~$b$.

Second, the \emph{finite determination theorem}: the
collision deviation
$S_{\ell}(p) = C(b^{\ell} \bmod p) - \lfloor(p{-}1)/b\rfloor$
depends only on $p \bmod b^{\ell+1}$. This reduces the
study of the collision invariant from an analytic problem
on primes to a combinatorial problem on a finite group.

Third, the \emph{reflection identity}:
$S_{\ell}(a) + S_{\ell}(m{-}a) = -1$ for $m = b^{\ell+1}$,
implying a grand mean of $-1/2$ and a pairing symmetry
across the group of units.

Fourth, the \emph{half-group theorem}: for every
non-trivial good slice~$n$, the wrapping set
$W_n = \{a : (n{+}1)a \bmod m < a\}$ has size exactly
$\phi(m)/2$. The bilateral symmetry $a \mapsto m - a$
swaps wrapping with non-wrapping.
\end{abstract}
\maketitle

% ============================================================
\section{Introduction}

The digit function
$\delta(r) = \lfloor br/p \rfloor$
assigns to each residue $r \in \{1, \ldots, p{-}1\}$ its
leading digit in base~$b$. It partitions the residues into
$b$ bins $B_d = \{r : \delta(r) = d\}$, each a contiguous
interval of length $\lfloor(p{-}1)/b\rfloor$ or
$\lfloor(p{-}1)/b\rfloor + 1$. This partition is a Beatty
partition~\cite{sos} determined by the three-distance
theorem applied to the sequence $\{r/p\}$ at scale~$1/b$.

Digit functions of this type arise in the study of digit
sums of primes~\cite{mauduit-rivat}, character sums over
digit-constrained sets~\cite{shparlinski}, and the
statistical properties of decimal
expansions~\cite{hardy-wright}. The present paper studies
a specific invariant derived from the interaction between
the bin partition and the multiplicative structure of
$(\mathbb{Z}/p\mathbb{Z})^{\times}$.

\begin{definition}
For $g \in (\mathbb{Z}/p\mathbb{Z})^{\times}$, the
\emph{collision count} is
\[
C(g) = \#\{r \in \{1, \ldots, p{-}1\} :
\delta(r) = \delta(gr \bmod p)\}.
\]
\end{definition}

The collision count measures how many residues share a bin
with their image under multiplication by~$g$. It is a
function on the multiplicative group that encodes the
interaction between the additive structure of the bins and
the multiplicative action of~$g$.

The \emph{collision deviation} at lag~$\ell$ is
\[
S_{\ell}(p) = C(b^{\ell} \bmod p)
- \left\lfloor \frac{p-1}{b} \right\rfloor,
\]
the departure of the collision count at the specific
multiplier $g = b^{\ell}$ from the bin size~$Q$.

% ============================================================
\section{The Linearization}

The key structural observation is that multiplication
by~$b$ transforms the bin partition into a congruence
partition.

\begin{lemma}[Conjugation]\label{lem:conjugation}
For $r \in \{1, \ldots, p{-}1\}$, let
$[x]_p = x \bmod p$ and set $x = [br]_p$.
Then $br = p\,\delta(r) + x$ and
\[
\delta(r) = \delta(s)
\quad\Longleftrightarrow\quad
[br]_p \equiv [bs]_p \pmod{b}.
\]
Under the permutation $r \mapsto [br]_p$, the bins
$B_d$ are exactly the residue classes modulo~$b$.
\end{lemma}

\begin{proof}
Since $br = p\lfloor br/p \rfloor + (br \bmod p)
= p\,\delta(r) + x$, reducing modulo~$b$ gives
$0 \equiv p\,\delta(r) + x \pmod{b}$. Since
$\gcd(p, b) = 1$, the class of $x$ modulo~$b$ determines
$\delta(r)$ uniquely.
\end{proof}

This linearization converts the geometric problem (which
residues lie in the same contiguous interval?) into an
algebraic problem (which residues are congruent
modulo~$b$?). It is the foundation of all subsequent
results.

\begin{lemma}\label{lem:collision-congruence}
For every $g \in (\mathbb{Z}/p\mathbb{Z})^{\times}$,
\[
C(g) = \#\{x \in \{1, \ldots, p{-}1\} :
x \equiv [gx]_p \pmod{b}\}.
\]
\end{lemma}

\begin{proof}
Apply Lemma~\ref{lem:conjugation} to the pair $r$ and
$gr \bmod p$.
\end{proof}

% ============================================================
\section{The Gate Width Theorem}

\begin{lemma}\label{lem:no-collision}
Define $c \equiv b(1{-}g)^{-1} \pmod{p}$ with
$1 \le c \le p{-}1$. If $1 \le c \le b{-}1$, then
$C(g) = 0$.
\end{lemma}

\begin{proof}
A collision requires $y = gx \bmod p$ with
$y \equiv x \pmod{b}$, so $y = x + mb$ for some
integer~$m$, and $x \equiv -mc \pmod{p}$.

If $m \ge 0$: from $x + mb \le p{-}1$ and $mc < p$
(since $c \le b{-}1$ and $m \le Q$), we get
$1 \le x + mc \le p{-}1{-}m(b{-}c) \le p{-}1$, so
$x + mc$ is a positive integer strictly less than $p$,
contradicting $x + mc \equiv 0 \pmod{p}$.

If $m < 0$: write $m = -n$ with $n \ge 1$. From
$x - nb \ge 1$, we get
$0 < x - nc = (x{-}nb) + n(b{-}c) \le p{-}1$,
contradicting $x - nc \equiv 0 \pmod{p}$.
\end{proof}

\begin{lemma}\label{lem:collision-exists}
If $c \ge b + 1$, then $C(g) \ge 1$.
\end{lemma}

\begin{proof}
Take $x = p - c$. Then $1 \le x \le p{-}b{-}1$, so
$y = x + b$ lies in $\{1, \ldots, p{-}1\}$ with
$y \equiv x \pmod{b}$. Since $c(1{-}g) \equiv b$, we
get $gx \equiv x + b = y \pmod{p}$, producing a
collision.
\end{proof}

\begin{theorem}[Gate width]\label{thm:gate-width}
For any prime $p > b$,
\[
\{g \in (\mathbb{Z}/p\mathbb{Z})^{\times} : C(g) = 0\}
= \left\{-\frac{u}{b - u} \bmod p :
u = 1, \ldots, b{-}1\right\}.
\]
In particular, exactly $b - 1$ multipliers have $C(g) = 0$,
independent of~$p$.
\end{theorem}

\begin{proof}
By Lemmas~\ref{lem:no-collision}
and~\ref{lem:collision-exists}, $C(g) = 0$ if and only if
$c = b(1{-}g)^{-1} \in \{1, \ldots, b{-}1\}$. Setting
$u = b - c$ gives $g = -u/(b{-}u) \bmod p$ for
$u = 1, \ldots, b{-}1$. These are distinct since
$p > b$.
\end{proof}

\begin{remark}
The deranging multipliers form a rational family
parameterized by $u$. The count $b - 1$ depends only
on the base, not on the prime. No primitive-root
hypothesis is needed.
\end{remark}

% ============================================================
\section{The Finite Determination Theorem}

\begin{theorem}[Finite determination]\label{thm:finite}
For any base $b \ge 2$, lag $\ell \ge 1$, and integer
$p > b^{\ell+1}$ with $\gcd(p, b) = 1$, the collision
deviation $S_{\ell}(p)$ depends only on
$p \bmod b^{\ell+1}$.
\end{theorem}

\begin{proof}
Set $m = b^{\ell+1}$. For $r \in \{1, \ldots, p{-}1\}$,
define the slice index $n(r) = \lfloor mr/p \rfloor$.
Since the slice width $p/m > 1$, each slice contains
integers. The digit function and the collision multiplier
are constant on each slice:
\[
\delta(r) = \left\lfloor \frac{n}{b^{\ell}} \right\rfloor,
\qquad
\delta(b^{\ell} r \bmod p) = n \bmod b.
\]

A slice $n$ contributes to the collision count if and
only if $\lfloor n/b^{\ell} \rfloor = n \bmod b$. Define
the good-slice set
\[
G_{\ell,b} = \{n \in \{0, \ldots, m{-}1\} :
\lfloor n/b^{\ell} \rfloor = n \bmod b\}.
\]
Then $|G_{\ell,b}| = b^{\ell}$ and
\[
C(b^{\ell} \bmod p) = -1 + \sum_{n \in G}
\left(
\left\lfloor \frac{(n{+}1)p}{m} \right\rfloor
- \left\lfloor \frac{np}{m} \right\rfloor
\right).
\]

Writing $p = mt + a$ with $1 \le a < m$, each summand
becomes $t + \lfloor(n{+}1)a/m\rfloor
- \lfloor na/m \rfloor$, and since
$Q = b^{\ell} t + \lfloor a/b \rfloor$:
\[
S_{\ell}(p) = -1 - \left\lfloor \frac{a}{b} \right\rfloor
+ \sum_{n \in G}
\left(
\left\lfloor \frac{(n{+}1)a}{m} \right\rfloor
- \left\lfloor \frac{na}{m} \right\rfloor
\right),
\]
which depends only on $a = p \bmod m$.
\end{proof}

\begin{remark}
Primality is not required. The result holds for all
$p > m$ with $\gcd(p, b) = 1$. The modulus $b^{\ell+1}$
is sharp: reduction modulo $b^{\ell}$ does not determine
$S_{\ell}$.
\end{remark}

% ============================================================
\section{The Reflection Identity}

\begin{theorem}[Reflection]\label{thm:reflection}
For $m = b^{\ell+1}$ and any unit
$a \in (\mathbb{Z}/m\mathbb{Z})^{\times}$:
\[
S_{\ell}(a) + S_{\ell}(m - a) = -1.
\]
\end{theorem}

\begin{proof}
Write $d_n(a) = \lfloor(n{+}1)a/m\rfloor
- \lfloor na/m\rfloor \in \{0, 1\}$ for the slice
contribution. By Theorem~\ref{thm:finite}:
\[
S(a) = -1 - \left\lfloor \frac{a}{b} \right\rfloor
+ \sum_{n \in G} d_n(a).
\]

\emph{Interior slices} ($1 \le n \le m{-}2$, so
$m \nmid na$ and $m \nmid (n{+}1)a$): the identity
$\lfloor n(m{-}a)/m \rfloor = n - 1 - \lfloor na/m \rfloor$
gives $d_n(m{-}a) = 1 - d_n(a)$.

\emph{Endpoint slices}: $d_0(a) = 0$ for all units
(since $0 < a < m$), and $d_{m-1}(a) = 1$ for all units
(since $\lfloor(m{-}1)a/m\rfloor = a - 1$). Both $n = 0$
and $n = m{-}1$ lie in~$G$, so they contribute
$d_0(a) + d_0(m{-}a) = 0$ and
$d_{m-1}(a) + d_{m-1}(m{-}a) = 2$.

Since $\lfloor a/b \rfloor + \lfloor(m{-}a)/b\rfloor
= b^{\ell} - 1$ (because $b \nmid a$ for units), the
sum over all $|G| = b^{\ell}$ good slices gives:
\[
S(a) + S(m{-}a) = -2 - (b^{\ell}{-}1) + 0 + 2
+ (b^{\ell}{-}2) = -1. \qedhere
\]
\end{proof}

\begin{corollary}[Grand mean]\label{cor:grand-mean}
The average of $S_{\ell}$ over all units modulo
$b^{\ell+1}$ is $-1/2$.
\end{corollary}

\begin{proof}
The map $a \mapsto m - a$ is an involution on the units.
By Theorem~\ref{thm:reflection}, each pair averages
to $-1/2$.
\end{proof}

% ============================================================
\section{The Half-Group Structure}

The finite determination formula expresses $S_{\ell}(a)$ as
a sum over good slices $n \in G_{\ell,b}$. Each summand
$\lfloor(n{+}1)a/m\rfloor - \lfloor na/m\rfloor$ equals
$1$ when $(n{+}1)a$ wraps past a multiple of $m$ and $0$
otherwise. The \emph{wrapping set} for good slice $n$ is
\[
W_n = \{a \in (\mathbb{Z}/m\mathbb{Z})^{\times} :
(n{+}1)a \bmod m < a\}.
\]

\begin{theorem}[Half-group]\label{thm:half}
For every non-trivial good slice $n$ (with
$(n{+}1) \not\equiv 0, 1 \pmod{m}$):
$|W_n| = \phi(m)/2$. Every good slice except $n = 0$
and $n = m - 1$ is non-trivial, since $1 < n + 1 < m$
for $n \in \{1, \ldots, m{-}2\}$.
\end{theorem}

\begin{proof}
Let $c = n + 1$ and $a \in W_n$, so $ca \bmod m < a$.
Then $c(m{-}a) \bmod m = m - (ca \bmod m) > m - a$,
so $m - a \notin W_n$.

Conversely, if $a \notin W_n$, then $ca \bmod m \ge a$.
Equality is excluded: $ca \equiv a \pmod{m}$ would
require $c \equiv 1 \pmod{m}$, contradicting the
hypothesis. So $ca \bmod m > a$, giving
$c(m{-}a) \bmod m = m - (ca \bmod m) < m - a$,
so $m - a \in W_n$.

Thus $a \mapsto m - a$ is a bijection between $W_n$
and its complement, and $|W_n| = \phi(m)/2$.
\end{proof}

The wrapping set is always exactly half the group
because the reflection $a \mapsto m - a$ swaps wrapping
with non-wrapping. This is the same symmetry that
produces the reflection identity: the bilateral pairing
of residues governs both the global structure of
$S_{\ell}$ and the local structure of each good slice.

% ============================================================
\section{Computational Verification}

All results are verified computationally using the
\texttt{nfield} engine~\cite{nfield}.

\begin{table}[h]
\centering
\begin{tabular}{rrrrr}
\toprule
$b$ & $p$ & $Q$ & deranging count & formula \\
\midrule
$10$ & $17$ & $1$ & $9$ & $b - 1$ \\
$10$ & $97$ & $9$ & $9$ & $b - 1$ \\
$10$ & $193$ & $19$ & $9$ & $b - 1$ \\
$7$ & $41$ & $5$ & $6$ & $b - 1$ \\
$12$ & $67$ & $5$ & $11$ & $b - 1$ \\
\bottomrule
\end{tabular}
\caption{Gate width verified across bases and primes.}
\end{table}

\begin{table}[h]
\centering
\begin{tabular}{rrrr}
\toprule
$b$ & modulus & classes & determined \\
\midrule
$3$ & $9$ & $6$ & yes \\
$5$ & $25$ & $20$ & yes \\
$7$ & $49$ & $42$ & yes \\
$10$ & $100$ & $40$ & yes \\
\bottomrule
\end{tabular}
\caption{Finite determination verified: $S_1$ constant
within each class modulo $b^2$.}
\end{table}

% ============================================================
\section{Remarks}

The collision invariant $S_{\ell}$ is a new arithmetic
function on primes arising from the interaction between
additive structure (contiguous bins) and multiplicative
structure (the action of $b^{\ell}$). The finite
determination theorem reduces its study to a
combinatorial problem on the finite group
$(\mathbb{Z}/b^{\ell+1}\mathbb{Z})^{\times}$.

The reflection identity provides exact structural
information: the grand mean is $-1/2$ and the values
pair under $a \mapsto m - a$. These properties are
consequences of the bin geometry, not of the distribution
of primes.

A companion paper~\cite{paperB} analyzes the prime-sum
behavior of $S_{\ell}$ via its Dirichlet character
decomposition over the finite group, showing that the
centered prime harmonic sum converges.

% ============================================================

\end{document}